\documentclass{article}
\usepackage{graphicx} 

\usepackage[english]{babel}
\usepackage[utf8]{inputenc}
\usepackage{mathtools}
\usepackage{rotating}
\usepackage{caption}

\usepackage{amsfonts}
\usepackage{amsmath}
\usepackage{amssymb}
\usepackage{booktabs}
\usepackage{multicol}
\usepackage{enumitem}
\usepackage{amsthm}
\usepackage{float}
\usepackage[margin=1.5in]{geometry}
\usepackage{hyperref}

\numberwithin{figure}{section}
\numberwithin{equation}{section}
\numberwithin{table}{section}
\newtheorem{theorem}{Theorem}[section]
\newtheorem{prop}[theorem]{Proposition}

\newtheorem{defn}[theorem]{Definition}

\newtheorem{example}[theorem]{Example}  
\newtheorem{cod}[theorem]{Condition}
\newtheorem{rmk}[theorem]{Remark}

\restylefloat{table}

\title{Greedy $\beta$-expansions for families of Salem numbers}
\author{Kevin G. Hare\footnote{University of Waterloo, Department of Pure Mathematics. Research of K. G. Hare is supported, in part, by NSERC Grant 2019-03930.} \\ email \href{mailto:kghare@uwaterloo.ca}{kghare@uwaterloo.ca} 
   \and Liam Orovec\footnote{University of Waterloo, Department of Pure Mathematics. Research of L. Orovec is supported, in part, by NSERC Grant 2019-03930.} \\ email \href{mailto:lorovec@uwaterloo.ca}{lorovec@uwaterloo.ca} }
\date{\today}

\begin{document}

\maketitle

\begin{abstract}
    We give criteria for finding the greedy $\beta$-expansion for $1$ for families of Salem numbers that approach a given Pisot number. We show that these expansions are related to the greedy expansion under the Pisot base. This expands on the work of Hare and Tweedle.
\end{abstract}

\section{Introduction}

Mathematicians have studied representations of numbers in bases other than 10 for hundreds of years; the study of non-integer bases, on the other hand, is relatively recent. It was first written about by Rényi \cite{renyi} in 1957. In the literature such objects are commonly referred to as $\beta$-expansions. Here we give a formal definition.

\begin{defn}
    Let $\beta \in (1,2)$. Consider the expansion
    \begin{equation}
        x=\sum_{j=1}^{\infty}a_j\beta^{-j}
        \label{expansion def}
    \end{equation}
    where $a_j\in\{0,1\}$. Then $a_1a_2\dots$ is a \textbf{$\beta$-expansion} for $x$. If there exists $k\in\mathbb{N}$ with $a_k=1$ such that $a_m=0$ for all $m\geq k$ then we say that the representation is \textbf{finite} and write it as $a_1a_2\dots a_k$.
\end{defn}

Unlike the case where the base is an integer, it is common to have multiple representations for the same real number. In such a case, we can define an ordering on all such representations.

\begin{defn}
    Let $\mathbf{a}=a_1 a_2\dots$ and $\mathbf{b}=b_1 b_2\dots$ be two different $\beta$-expansions for a real number $x$. Then we say that $\mathbf{a}$ is \textbf{lexicographically less} than $\mathbf{b}$ if there exists $k\in\mathbb{N}$ such that $a_k<b_k$ and $a_i=b_i$ for all $i<k$. If this is the case, we write $\mathbf{a}<_{lex}\mathbf{b}$. When comparing a finite expansion, we use the equivalent infinite expansion by appending an infinite number of zeros to the end of the sequence.
\end{defn}

\begin{example}
    Consider the expansions of 1 under base $\phi\approx 1.61803$, the Golden Ratio and root of $x^2-x-1$. It is known that the finite representations of $1$ in base $\phi$ can be written as $1(01)^n1$ for each $n\in\mathbb{Z}_{\geq 0}$. The infinite expansion $1(01)^{\omega}$ is also valid. It is easy to see from the above definition that we have $11 >_{lex} 1011 >_{lex} 1(01)^21 >_{lex} \cdots >_{lex} 1(01)^{\omega}>_{lex} 01^{\omega}$.
\end{example}

We could instead define $\beta$-expansions for any $\beta>1$ by taking $a_i\in\{0,1,\dots,\lfloor\beta\rfloor\}$. However, restricting ourselves to $\beta\in(1,2)$ is particularly useful when looking at what is known as the greedy expansion.

\begin{defn}
    If $a_1a_2\dots$ is the maximal lexicographic expansion $\beta$ for $x$ with base $\beta$, then we say that $a_1a_2\dots$ is the \textbf{greedy expansion} for $x$ with base $\beta$. This is denoted by $d_{\beta}(x)$. 
\end{defn}

One can describe an algorithm which for a given value $x$ and base $\beta$ produces the greedy expansion of $x$ in base $\beta$. By restricting to $\beta\in(1,2)$, we ensure that this algorithm has only two possibilities to output a given digit. We also consider the lexicographically largest infinite representation, which is called the \textbf{quasi-greedy expansion} and write $d_{\beta}^*(x)$. For example, the Golden Ratio $\phi\approx 1.61803$ has greedy $\beta$-expansion for $1$ given by the sequence $d_{\beta}(1)=11$. If we apply the greedy algorithm while ensuring that the representation is not finite, we find that $1(01)^{\omega}$ is another valid $\beta$-expansion for $1$ and is the largest, which is not finite. See \cite{komo} for a more detailed construction and further results on the subject.

\begin{defn}
    If $a_1a_2\dots$ is a $\beta$-expansion and suppose that there exists $k,\ell\in\mathbb{N}$ such that $a_{n}=a_{n+\ell}$ for all $n>k$ we say that $a_1a_2\dots$ is \textbf{eventually periodic} and we write it as $a_1\dots a_k(a_{k+1}\dots a_{k+\ell})^{\omega}$.
    We call the string $a_1\dots a_k$ the preperiodic part of the expansion and we say the preperiod has length $k$. The string $a_{k+1}\dots a_{k+\ell}$ is called the periodic part of the expansion and has length $\ell$.
\end{defn}

We will primarily be interested in $\beta$-expansions for $1$ that are finite or eventually periodic. In such a case, we can define a polynomial whose coefficients are related to the said expansion. In particular, if $1$ has a finite greedy $\beta$-expansion given by $a_1\dots a_k$, then we can write the quasi-greedy $\beta$-expansion as $a_1\dots a_{k-1}(a_k-1)(a_1\dots a_{k-1}(a_k-1))^{\omega}$, which can be viewed as having either an empty preperiod or a preperiod of length $k$.

\begin{defn}
    Consider the greedy expansion $d_{\beta}(1)=a_1a_2\dots a_k$ if the greedy expansion is finite and $d_{\beta}(1)=a_1a_2\dots a_k(a_{k+1}\dots a_{k+\ell})^{\omega}$ if the greedy expansion is eventually periodic. Define $P_j(x)=x^j-a_1x^{j-1}-\cdots-a_j$. The \textbf{companion polynomial} is defined as
    \begin{equation}
        R(x) = \begin{cases} 
          P_k(x) & \text{finite case} \\
          P_{k+\ell}(x)-P_k(x) &  \text{periodic case.}
       \end{cases} 
       \label{companion}
    \end{equation}
\end{defn}    

Note that we need not take $k$ to be the minimal length of the preperiod, the resulting polynomial will have an extra factor of the form $x^m$, where $m$ is the difference between the length of the preperiod used and that of minimal length.

\begin{rmk}
    In both cases, we have $R(\beta)=0$. Together with the fact that $R(x)$ is a monic integer polynomial, it follows that if the greedy $\beta$-expansion for 1 is finite or eventually periodic, then $\beta$ is an algebraic integer.
\end{rmk}
   
\begin{defn}   
   If $\beta$ has a minimal polynomial $M(x)$, we have $M(x)|R(x)$. We write $R(x)=M(x)Q(x)$. We call $Q(x)$ \textbf{co-factor} of the $\beta$-expansion. 
\end{defn}

We will be looking at specific sequences and we will use the following theorem to help determine if they are greedy $\beta$-expansions for 1.

\begin{theorem}[Parry \cite{parry}]
    Let $\mathbf{a}=(a_n)_{n\geq 1}$ be a sequence in $\{0,1\}^{\mathbb{N}}$. That is, $\mathbf{a}$ is a non-empty finite or infinite sequence $a_1\dots a_k$ or $a_1a_2\dots$ with $a_i\in\{0,1\}$. Then the sequence $\mathbf{a}$ is the greedy expansion of $1$ for some $\beta>1$ if and only if for all $j\geq 1$
    $$\sigma^j(\mathbf{a})<_{\text{lex}}\mathbf{a}.$$
    Here $\sigma(a_1a_2\dots)=a_2a_3\dots$. 
    \label{greedy}
\end{theorem}

If $\mathbf{a}=a_1a_2\dots a_k(a_{k+1}\dots a_{k+\ell})^{\omega}$, then it suffices to check the above condition only with $1\leq j\leq k+\ell$.

We define two families of algebraic integers that often have particularly nice greedy expansions.

\begin{defn}
    A \textbf{Pisot number} $q$ is a real algebraic integer $q>1$ such that all conjugates of $q$ are strictly less than $1$ in modulus.
\end{defn}

There is a great deal of study on Pisot numbers and their structure. For example, the real root of $x^3-x-1$ is proven to be the smallest Pisot number \cite{siegel}. In Section 4, we will discuss what is known about the limit points of the set of Pisot numbers within the interval $(1,2)$.

We will introduce the notion of regular Pisot numbers in Section 4, these have significant structure in the interval $(1,2)$. An extensive study has been done on greedy expansions for $1$ under such bases. In fact, we have a complete description of the greedy expansion of $1$ under any regular Pisot base, which has been established by Panju \cite{panju}. 

The next family of algebraic integers is the primary concern of these results and they are closely related to Pisot numbers.

\begin{defn}
    A \textbf{Salem number} $\alpha$ is a real algebraic integer $\alpha>1$ such that all conjugates of $\alpha$ are less than or equal to $1$ in modulus and at least one conjugate is equal to $1$ in modulus.
\end{defn}

Comparatively, much less is known in reference to Salem numbers than about the Pisot numbers. For example, the smallest known Salem number is the real root of $x^{10}+x^9-x^7-x^6-x^5-x^4-x^3+1+1$, yet no proof that it is the smallest in the set has been widely accepted.

If $M(x)$ is the minimal polynomial of a Pisot number, then for sufficiently large $m$, $T_m^{\pm}(x)=M(x)x^m\pm M^*(x)$ admits a Salem number as a root. Here, $M^*(x)=x^{\operatorname{deg}(M)}M(1/x)$ is the reciprocal polynomial of $M(x)$. Moreover, as $m\rightarrow\infty$, these Salem numbers approach the Pisot number as a two-sided limit. Note that $T_m^{\pm}(x)$ may not be the minimal polynomial for the Salem number, as it may have cyclotomic factors. In fact, in the minus case, $T_m^-(x)$ always has at least $(x-1)$ as a factor.

An important result on Pisot numbers states that if $q$ is a Pisot number, then the greedy expansion of $1$ in base $q$ is either finite or eventually periodic \cite{bertrand}. This is not true for a general algebraic integer; for example, $\beta=\sqrt{2}$ does not admit a finite or periodic expansion. To see this, we note that if it did have such an expansion then $(x^2-2)|R(x)$. However, $R(x)$  has a tail coefficient of $\pm 1$, as seen by Definition \ref{companion}. We also know that all Salem numbers of degree 4 admit periodic $\beta$-expansions for 1, and it is conjectured that the same holds for degree 6 but not in general, \cite{boyd1} and \cite{boyd2}.

It is not a sufficient condition to have a periodic $\beta$-expansion for 1 to guarantee that the base is a Pisot or Salem number. For example $\beta\approx 1.7403$, the root of $x^8-x^7-x^6-x^3-x-1$ has $d_{\beta}(1)=11001011$ but is not Pisot or Salem. The following result is known to be true.

\begin{theorem}{(Schmidt, 1980 \cite{schmidt})}
    If $d_{\beta}(x)$ is periodic or finite for all $x\in\mathbb{Q}\cap[0,1)$, then $\beta$ is a Pisot or Salem number.
\end{theorem}
Schmidt also conjectured that the converse of this theorem holds.

In what follows, we study the greedy expansions for 1 where the base is a Pisot or Salem number. We give conditions to be able to determine when a Salem number has periodic expansion, this expansion is given in terms of the Pisot limit that the family of Salem numbers are approaching. Toward this goal, we consider the following condition, which is well defined. It is an expansion of the condition considered by Hare and Tweedle, \cite{hareTweedle}.

\begin{cod}
    We say an algebraic integer $q$ with $d_q(1)=a_1\dots a_k(a_{k+1}\dots a_{k+\ell})^{\omega}$ has a \textbf{reversibly greedy} $\beta$-expansion if it is greedy and;
    $$a_1\dots a_{k+\ell}>_{\text{lex}}a_{k+\ell-i}a_{k+\ell-i-1}\dots a_2\;\text{ for all } i, \text{ with } 0\leq i\leq k+\ell-2.$$
    \label{revgre}
\end{cod}

Here we replace finite sequences by the sequence of appropriate length by appending a number of 0's to the end as needed. In general, we will want to consider the cases where the given algebraic integer is a Pisot number. Note that for future reference, if instead we use $d_q^*(1)$, the quasi-greedy expansion, then this condition is satisfied only if it is for $d_q(1)$.

\begin{prop}
    Suppose $a_1\dots a_k$ is a finite reversibly greedy $\beta$-expansion. Then $(b_1\dots b_k)^{\omega}$ satisfies Condition \ref{revgre}, where $b_i=a_i$ for $i=1,\dots, k-1$ and $b_k=a_k-1$.
    \label{quasi}
\end{prop}

\begin{proof}
    Let $\mathbf{c_i}=b_{k-i}\dots b_2$ and $\mathbf{b}=b_1\dots b_k$. We need to show $\mathbf{b}>_{\text{lex}}\mathbf{c_i}$ for each $i\leq k-2$. The length of the string $\mathbf{b_i}$ is less than $k$, hence we can replace $\mathbf{b}$ by the string $b_1\dots b_{k-1}$. We can do the same for the analogous strings $\mathbf{a}=a_1\dots a_k$ and $\mathbf{d_i}=a_{k-i}\dots a_2$.

    We have $a_1\dots a_{k-1}=b_1\dots b_{k-1}$ and $\mathbf{a_i}>_{\text{lex}}\mathbf{b_i}$, the result follows.
\end{proof}

Given a Pisot number, $\beta$, with defining polynomial $M(x)$ and a natural number $m$ we aim to give the $\beta$-expansion for the related Salem numbers $\alpha_m^{\pm}$ defined by $M(x)x^m\pm M^*(x)$.

In Section 2 we review previous work by Hare and Tweedle \cite{hareTweedle} which covers the positive case where $\beta$ is a Pisot number with finite greedy expansion for 1. Further, we expand these results to the case where the expansion is eventually periodic. We give sufficient criteria to determine the greedy expansion for 1 under certain Salem number bases related to $\beta$.

Section 3 considers the negative cases, which is once again split between the eventually periodic and finite cases. We explain in this section that the finite case is a direct consequence of the eventually periodic case when we consider the quasi-greedy expansion.

Section 4 includes the motivation behind the results and gives some infinite family of Pisot numbers that satisfy the conditions of the theorems of Sections 2 and 3. These results show that the conditions of these theorems are relatively common occurrences. Finally, Section 5 mentions some further work that could be done on this topic.

\section{The Positive Case}

\subsection{Finite Case}

Hare and Tweedle \cite{hareTweedle} considered the case where $q$ had a finite $\beta$-expansion for $1$. They were able to determine a relationship between the $\beta$-expansion for the Pisot number and the $\beta$-expansions for the related Salem numbers defined by $M(x)x^m+M^*(x)$.

\begin{theorem}{(Hare \& Tweedle, 2008)}
    Let $q$ be a Pisot number with minimal polynomial $M(x)$. Further, let $q$ have a finite reversibly greedy $\beta$-expansion, say $a_1a_2\dots a_k$. In addition, assume that the companion polynomial of $q$ has co-factor polynomial $Q(x)$, which is reciprocal. Let $\alpha_m^+$ be the Salem number that satisfies the polynomial $M(x)x^m+M^*(x)$. Then for $m>2k$ we have the $\beta$-expansion of $\alpha_m^+$ is
    $$\mathbf{a_m^+}=a_1(a_2\dots a_k0^{m-k-1}a_k\dots a_200)^{\omega}.$$
\end{theorem}

They prove this by showing two sufficient criteria. The first is that the polynomial $M(x)x^m+M^*(x)$ divides the companion polynomial of the expansion, $a_1(a_2\dots a_k0^{m-k-1}a_k\dots a_200)^{\omega}$. This is sufficient to show that the expansion is an $\alpha_m^+$-expansion of $1$. This is because the minimal polynomial of the Salem number $\alpha_m^+$ divides $M(x)x^m+M^*(x)$ by construction. They then used Theorem 1.4 to conclude that the expansion is, in fact, a greedy expansion for $1$ under some base. Combining these two results shows that it is the greedy expansion for $1$ under base $\alpha_m^+$.

To aid in our future cases, we can rewrite this case using the following notation. We let $\gamma=a_1a_2\dots a_k$ be the reversibly greedy $q$-expansion for $1$ and let $\kappa = a_2\dots a_k=\sigma(\gamma)$ be the same sequence beginning with the second entry.

Then the theorem states that the greedy $\alpha_m^+$-expansion for $1$ is given by,
$$\mathbf{a_m^+}=1(\kappa 0^{m-k-1}\kappa^*00)^{\omega}.$$

Here $\kappa^*=a_k\dots a_2$ is the inverted string. Note that this result holds for all finite reversibly greedy Pisot numbers and sufficiently large $m$, in what follows our results will be conditional on further properties to be discussed.

\subsection{Periodic Case}

Let $\beta$ be a Pisot number with minimal polynomial $M(x)$. Suppose that the greedy $\beta$-expansion for $1$ is given by $a_1\dots a_k(a_{k+1}\dots a_{k+\ell})^{\omega}$ and that it satisfies Condition 1.8. Let $R(x)$ and $Q(x)$ be the companion polynomial and co-factor of this expansion, respectively, so that $R(x)=M(x)Q(x)$, where $R(x)$ satisfies Equation \ref{companion}. We aim to show a relationship between the greedy $\beta$-expansion and the greedy expansions of the Salem numbers $\alpha_m^+$ which are the roots of the polynomials $T_m(x)=M(x)x^m+M^*(x)$.

\begin{example}
    Considering the Pisot number with the defining polynomial $M(x)=x^4-x^3-2x^2+1$, we know that it has $beta$-expansion given by $11(10)^{\omega}$. For sufficiently large values of $m$ we can define the polynomial $T_m(x)=M(x)x^m+M^*(x)$ which yields a Salem number. We look at the beta-expansions of successive Salem numbers, see Table \ref{Salem}.

\begin{table}[h]
    \centering
    \begin{tabular}{l|l|l}
        $m$ & $\beta$-expansion & $\beta$-expansion\\ \hline
        4 & $1(110001100)^{\omega}$ & $1(1(10)0(01)100)^{\omega}$ \\
        5 & $1(110011001100)^{\omega}$ & $1(1(10)0110(01)100)^{\omega}$ \\
        6 & $1(1101000101100)^{\omega}$ & $1(1(10)^20(01)^2100)^{\omega}$ \\
        7 & $1(1101001100101100)^{\omega}$ & $1(1(10)^20110(01)^2100)^{\omega}$ \\
        8 & $1(11010100010101100)^{\omega}$ & $1(1(10)^30(01)^3100)^{\omega}$ \\
        9 & $1(11010100110010101100)^{\omega}$ & $1(1(10)^30110(01)^3100)^{\omega}$ \\
        10 & $1(110101010001010101100)^{\omega}$ & $1(1(10)^40(01)^4100)^{\omega}$ \\
        11 & $1(110101010011001010101100)^{\omega}$ & $1(1(10)^40110(01)^4100)^{\omega}$ \\
    \end{tabular}
    \caption{$\beta$-expansion for some Salem Numbers}
    \label{Salem}
\end{table}

    We observe a pattern. The even values of $m$ give an expansion of the form $1(1(10)^d0(01)^d100)^{\omega}$ where $m=2(d+1)$ and the odd values of $m$ give an expansion of the form $1(1(10)^d0110(01)^d100)^{\omega}$ where $m=2(d+1)+1$.
    \label{first example}
\end{example}

This is an example of a more general phenomenon that we will prove with the following theorem.

\begin{theorem}
    Let $q$ be a Pisot number with minimal polynomial $M(x)$. Further, let $q$ have an infinite reversibly greedy $\beta$-expansion, say $a_1a_2\dots a_k(a_{k+1}\dots a_{k+\ell})^{\omega}$. Next we assume that the companion polynomial of $q$ has co-factor polynomial $Q(x)$, which is reciprocal. Let $\alpha_{n\ell+j}^+$ be the Salem number that satisfies the polynomial $M(x)x^{n\ell+j}+M^*(x)$ for each $j=1,\dots,\ell$. We let $\gamma=a_2\dots a_k$ and $\kappa=a_{k+1}\dots a_{k+\ell}$.
    
    Suppose that the greedy $\beta$-expansion for $1$ in base $\alpha_{\ell+j}^+$ is given by
    \begin{equation}
        \mathbf{a_{(1,j)}}=a_1(a_2\dots a_k\tau_ja_k\dots a_200)^{\omega}=1(\kappa\tau_j\kappa^*00)^{\omega},
    \end{equation}

    for some string $\tau_j$ of length $3\ell-k-1+j$.

    Then for $n\geq 2$ the greedy $\beta$-expansion for $1$ in base $\alpha_{n\ell+j}^+$ is given by
    \begin{align}
        \mathbf{a_{(n,j)}}&=a_1(a_2\dots a_k(a_{k+1}\dots a_{k+\ell})^{n-1}\tau_j(a_{k+\ell}\dots a_{k+1})^{n-1}a_k\dots a_200)^{\omega}\nonumber\\
        & =1(\kappa\underbrace{\gamma\gamma\dots\gamma}_{n-1}\tau_j\underbrace{\gamma^*\gamma^*\dots\gamma^*}_{n-1}\kappa^*00)^{\omega}.
    \end{align}
    \label{MAIN}
\end{theorem}

\begin{rmk}
    We can write Example \ref{first example} in the terms of this theorem. As the $\beta$-expansion for the Pisot number is given by $11(10)^{\omega}$ we have $\kappa=1$, $\gamma=10$, $\ell=2$ and we find $\tau_1=0110$ and $\tau_2=10001$.
\end{rmk}

\begin{proof}[Proof of Theorem \ref{MAIN}]
    We will show this by proving that the polynomial $R_{n\ell+j}(x)$ defined below is the companion polynomial for the $\beta$-expansion of $1$ in base $\alpha_{n\ell+j}$ under the assumption that $R_{\ell+j}(x)$ is the companion polynomial for the expansion of $1$ in base $\alpha_{\ell+j}$.

    Let $R_{n\ell+j}(x)=Q(x)(M(x)x^{n\ell+j}+M^*(x))(1+x^{\ell}+\cdots+x^{\ell n})$ and let $R(x)$ be the companion polynomial of the expansion of $q$, the original Pisot base. That is $R(x)=M(x)Q(x)$. Let $P_k(x)$ and $P_{k+\ell}(x)$ be the constituent parts of $R(x)$ and define $S(x)$ such that $P_{k+\ell}(x)=x^{\ell}P_k(x)-S(x)$. That is
    \begin{equation}
        R(x)=P_{k+\ell}(x)-P_k(x),\quad R(x)^*=P_{k+\ell}^*(x)-x^{\ell}P_k^*(x).
        \label{T1}
    \end{equation}
    and,
    \begin{equation}
        P_{k+\ell}(x)=x^{\ell}P_k(x)-S(x),\quad P_{k+\ell}^*(x)=P_k^*(x)-x^{k+1}S^*(x).
        \label{T2}
    \end{equation}
    Note that $S(x)$ is considered a degree $\ell-1$ polynomial, in order to calculate $S^*(x)$. Observe that because $Q(x)$ is a reciprocal co-factor we have $R(x)=Q(x)M(x)$ and $R^*(x)=Q(x)M^*(x)$.

    We will first show that $R_{n\ell+j}(x)$ must be the companion polynomial for a $\beta$-expansion of the Salem number $\alpha_{n\ell+j}$. We begin by first rewriting $R_{n\ell+j}(x)$ as two sums involving $R(x)$ and $R^*(x)$ respectively.

    \begin{align}
        R_{n\ell+j}(x) &= (Q(x)M(x)x^{n\ell+j}+Q(x)M^*(x))(1+x^{\ell}+\cdots+x^{\ell n})\\
        &= (R(x)x^{n\ell+j}+R^*(x))(1+x^{\ell}+\cdots+x^{\ell n})\\
        &= \sum_{i=0}^{n}R(x)x^{j+\ell(n+i)} + \sum_{i=0}^{n}R^*(x)x^{\ell i}.
    \end{align}
    
    Since $R(x)$ is the companion polynomial of the beta-expansion under the Pisot base we can separate the summands into the components $P_{k+\ell}(x), P_k(x), P^*_{k+\ell}(x)$ and $P^*_k(x)$. Further splitting the sum using the identities (\ref{T1}) and (\ref{T2}) we begin to see the expected terms coming from the conjectured companion polynomial.
    
    \begin{align}
        R_{n\ell+j}(x) =&\; \left(\sum_{i=0}^{n}P_{k+\ell}(x)x^{j+\ell(n+i)}-P_k(x)x^{j+\ell(n+i)}\right)\nonumber \\
        &+ \left(\sum_{i=0}^{n}P_{k+\ell}^*(x)x^{\ell i}-P_k^*(x)x^{\ell(i+1)}\right)\\
        =&\;\left(P_{k+\ell}(x)x^{j+2n\ell}+\sum_{i=0}^{n-1}(x^{\ell}P_k(x)-S(x))x^{j+\ell(n+i)}-\sum_{i=0}^{n}P_k(x)x^{j+\ell(n+i)}\right)\nonumber\\
        &+ \left(P_{k+\ell}^*(x) + \sum_{i=1}^{n}(P_k^*(x)-x^{k+1}S^*(x))x^{\ell i}-\sum_{i=1}^{n+1}P_k^*(x)x^{\ell i}\right).
    \end{align}
    
    Finally by rearranging the summations we find three terms. The first two which the reciprocal to each other and, the final which is exactly the polynomial with coefficients that are exactly the digits in $\tau_j$. This is by the assumption that $R_{\ell+j}(x)$ is the companion polynomial of $\mathbf{a}_{(1,j)}$.
    
    \begin{align}   
        R_{n\ell+j}(x) =&\;\left(P_{k+\ell}(x)x^{j+2n\ell}-P_k(x)x^{j+\ell n}-\sum_{i=0}^{n-1}S(x)x^{j+\ell(n+i)}\right)\nonumber\\
        & + \left(P_{k+\ell}^*(x)-P_k^*(x)x^{\ell(n+1)} - \sum_{i=1}^{n}S^*(x)x^{\ell i+k+1} \right)\\
        =&\; \left(P_k(x) x^{j+(2n+1)\ell}-\sum_{i=2}^{n}S(x)x^{j+\ell(n+i)}\right) + \left(P_k^*(x) - \sum_{i=0}^{n-2}S^*(x)x^{\ell i+k+1}\right)\nonumber \\
        &-\left((P_k(x)+S(x))x^{j+\ell n}+S(x)x^{j+\ell(n+1)}\right)\nonumber \\
        &+\left(P_k^*(x)x^{\ell (n+1)}+S^*(x)x^{\ell(n-1)+k+1}+S^*(x)x^{\ell n+k+1}\right).
    \end{align}

    Under our assumption we must have the final two terms of the above polynomial are related only to the string of digits $\tau_j$, note that this polynomial has a range of degrees of length $|\tau_j|=3\ell-k-1+j$. The leading term represents the part of the companion polynomial preceding $\tau_j$  in the conjecture and the second term is exactly the part of the companion polynomial referring to the part following $\tau_j$ in the conjectured. We also notice that each of the terms in the above sum have independent monomial terms, hence there is no interaction between the sections.\\

    Therefore $R_{n\ell+j}(x)$ as defined is the companion polynomial of the expansion in question. Noting that $M(x)x^{n\ell+j}+M^*(x)$ divides $R_{n\ell+j}(x)$ we conclude that $\mathbf{a}$ is indeed a valid $\beta$-expansion for 1 in base $\alpha_{\ell n+j}^+$. We also see that the co-factor of the expansion is given by $Q(x)(1+x^{\ell}+\cdots+x^{\ell(n-1)})$. It remains to show it is the greedy expansion.\\

    It suffices to apply Theorem \ref{greedy} to complete the proof. We have $\sigma^i(\mathbf{a_{(n,j)}}) < \mathbf{a_{(n,j)}}$ for $i=1,\dots k+\ell(n-1)$ since $a_1\dots a_k(a_{k+1}\dots a_{k+\ell})^{\omega}$ is a greedy expansion of $q$. Next $\sigma^i(\mathbf{a_{(n,j)}}) < \mathbf{a_{(n,j)}}$ for $i=k+\ell(n-1)+1\dots k+\ell(n-1)+|\tau_j|$ since the relation holds for $\mathbf{a_{(1,j)}}$ by assumption. Finally we see that $\sigma^j(\mathbf{a_{(n,j)}})< \mathbf{a_{(n,j)}}$ for the remaining terms as $a_1\dots a_k(a_{k+1}\dots a_{k+\ell})^{\omega}$ is reversibly greedy.  
\end{proof}

We have shown the theorem to be true; it remains to explore how often the conditions of the theorem actually occur.

\section{The Negative Case}

\subsection{Periodic Case}
We now turn our attention to the negative case and explore the $\beta$-expansions of Salem numbers and when they are eventually periodic. Let $M(x)$ be the minimal polynomial of a Pisot number and consider the Salem number defined by the polynomial $T_m^-(x)=M(x)x^m-M^*(x)$.

\begin{example}
    Consider the Pisot root $q$ of $M(x)=x^6-x^5-x^4-x^2+1$. We calculate the greedy $\beta$-expansions for the Salem numbers defined by $M(x)x^m-M^*(x)$ for certain $m$ in the following table. The Pisot number has a greedy $\beta$-expansion for $1$ given by $d_q(1)=11(0010)^{\omega}$. Let $\kappa = 1$ and $\gamma = 0010$. Note that in general the expansions for $m=4n+1$ and $m=4n+3$ do not appear to be periodic when calculating up to periods of length 100.

    \begin{table}[h]
        \centering
        \begin{tabular}{l|l|l}
            $m$ & Expansion & Expansion\\ \hline
            12 & $1(\kappa\gamma\gamma0\kappa^*00)^{\omega}$ & $1(\kappa\gamma\gamma010010\gamma^*\gamma^*\kappa^*00)^{\omega}$\\
            16 & $1(\kappa\gamma\gamma\gamma0\kappa^*00)^{\omega}$ & $1(\kappa\gamma\gamma\gamma010010\gamma^*\gamma^*\gamma^*\kappa^*00)^{\omega}$\\
            20 & $1(\kappa\gamma\gamma\gamma\gamma0\kappa^*00)^{\omega}$ & $1(\kappa\gamma\gamma\gamma\gamma010010\gamma^*\gamma^*\gamma^*\gamma^*\kappa^*00)^{\omega}$\\ \hline
            6 & $1(\kappa\gamma1\gamma^*\kappa^*00)^{\omega}$\\
            10 & $1(\kappa\gamma\gamma1\gamma^*\gamma^*\kappa^*00)^{\omega}$\\
            14 & $1(\kappa\gamma\gamma\gamma1\gamma^*\gamma^*\gamma^*\kappa^*00)^{\omega}$\\ \hline
        \end{tabular}
        \caption{$\beta$-expansions for Salem numbers related $x^6-x^5-x^4-x^2+1$}
        \label{neg_exp}
    \end{table}

    Unlike the positive case, we can see two different patterns which emerge. We will prove that in the case that an expansion of the above form appears, the pattern will continue indefinitely. Note that here $\kappa$ is the pre-periodic part of the greedy $q$-expansion for $1$ without the leading digit. Notice that it appears that the first type of pattern is of the same form as the second when we do not use minimal period lengths. We will later explain that this is not generally true, it does not follow exactly in this case. We note that as in the positive case, the repeating pattern continues with a period of length equal to the length of the periodic part of the greedy $q$-expansion.
\end{example}

Note that similarly to the positive periodic case, we see that some residues modulo $\ell$ do appear to ever have either of the given patterns.

\begin{theorem}
    Let $q$ be a Pisot number with minimal polynomial $M(x)$. Suppose that $1$ has an infinite reversibly greedy $\beta$-expansion in base $q$, say $a_1a_2\dots a_k(a_{k+1}\dots a_{k+\ell})^{\omega}$ with $\ell$ even. Let $\ell=2p$, $\kappa=a_2\dots a_{k}$ and $\gamma=a_{k+1}\dots a_{k+2p}$.
    
    Furthermore, assume that the companion polynomial of $q$ has co-factor polynomial $Q(x)$, which is reciprocal. Let $\alpha_{2np+j}^-$ be the Salem number that satisfies the polynomial $M(x)x^{2np+j}-M^*(x)$ for some $j\geq 1$.
    
    Suppose that the $\beta$-expansion for $\alpha_{2p+j}^-$ is given by
    \begin{align}
       \mathbf{a_{(1,j)}}&=a_1(a_2\dots a_{k}(a_{k+1}\dots a_{k+2p})\tau_j(a_{k+2p}\dots a_{k+1})a_{k}\dots a_200)^{\omega}\nonumber\\
       &=1(\kappa\gamma\tau_j\gamma^*\kappa^*00)^{\omega},
    \end{align}

    for some string $\tau_j$ of length $j-k+p-1$.

    Then for $n\geq 1$ the $\beta$-expansion for $\alpha_{2np+j}^-$ is given by
    \begin{align}
        \mathbf{a_{(n,j)}}&=a_1(a_2\dots a_{k}(a_{k+1}\dots a_{k+2p})^{n}\tau_j(a_{k+2p}\dots a_{k+1})^{n}a_{k}\dots a_200)^{\omega}\nonumber\\
        &=1(\kappa\underbrace{\gamma\gamma\dots\gamma}_{n}\tau_j\underbrace{\gamma^*\gamma^*\dots\gamma^*}_{n}\kappa^*00)^{\omega}.
    \end{align}
    \label{MAIN NEG INF}
\end{theorem}

\begin{proof}
    Consider the following polynomial;
    \begin{equation}
        R_{2np+j}(x)=Q(x)(M(x)x^{2np+j}-M^*(x))\left(\frac{x^{2np+p}+1}{x^{2p}-1}\right).
        \label{R(x)Neg1.1}
    \end{equation}
    We show that if $R_{2p+j}(x)$ is the companion polynomial to the $\beta$-expansion of $\alpha_{2p+j}^-$ then $R_{2np+j}(x)$ must be the companion polynomial for the $\beta$-expansion of $\alpha_{2np+j}^-$. In doing so, we must also prove that $R_{2np+j}(x)$ is an integer polynomial. Note that if $\ell=2p$ were not even then this would not be an integer polynomial in general.

    We need to show that the minimal polynomial of $\alpha_{2np+j}^-$, which divides $M(x)x^{2np+j}-M^*(x)$, also divides $R_{2np+j}(x)$. And that it is the companion polynomial of the representation in the statement of the theorem. Toward this conclusion, let $R(x)=Q(x)M(x)$ be the companion polynomial of the $\beta$-expansion for $1$ under base $q$, the Pisot number in question.

    First note that 
    \begin{equation}
       \frac{x^{2np+p}+1}{x^{2p}-1}=\frac{1-x^{p}+x^{2p}-x^{3p}+x^{4p}-\dots-x^{2np-p}+x^{2np}}{x^{p}-1}. 
    \end{equation}
    This rational function has a numerator which can be written as,
    \begin{equation}
        f(x)=(x^{p}-1)(1+x^{2p}+\dots+x^{2(n-1)p})(x^{p}-1)+g(x)
    \end{equation}
    where,
    \begin{equation}
        g(x)=x^{p}-x^{2p}+x^{3p}-x^{4p}+\dots+x^{2np-p}.
    \end{equation}

    Then using this representation we have,
    
    \begin{align}
        R_{2np+j}(x)=& \;(R(x)x^{2np+j}-R^*(x))\left(\frac{x^{2np+p}+1}{x^{2p}-1}\right) \label{R(x)Neg1.2}\\
        =& \; x^{2np+j+p}R(x)(1+x^{2p}+\dots+x^{2(n-1)p})+R^*(x)(1+x^{2p}+\dots+x^{2(n-1)p})\nonumber\\
        & -x^{2np+j}R(x)(1+x^{2p}+\dots+x^{2(n-1)p})-x^{p}R^*(x)(1+x^{2p}+\dots+x^{2(n-1)p})\nonumber\\
        & +(R(x)x^{2np+j}-R^*(x))\left(\frac{g(x)}{x^{p}-1}\right). \label{R(x)Neg1.3}
    \end{align}

    We will now focus on the final term of the polynomial in this form, that is, 
    \begin{equation}
        h(x)=(R(x)x^{2np+j}-R^*(x))\left(\frac{g(x)}{x^{p}-1}\right)
        \label{h(x)Neg1.1}
    \end{equation}
    We can write $g(x)$ in the following two forms,
    \begin{align}
        g(x)&=(x^{p}-1)(1+x^{2p}+\dots+x^{2(n-1)p})+\frac{1}{x^{p}-1}\\
            &=-x^{p}(x^{p}-1)(1+x^{2p}+\dots+x^{2(n-1)p})+\frac{x^{2np}}{x^{p}-1}.
    \end{align}

    Using the appropriate representation for $g(x)$ as we expand $h(x)$ we find,
    \begin{align}
        h(x) =&\; (R(x)x^{2np+j}-R^*(x))\left(\frac{g(x)}{x^{p}-1}\right)\\
        =&\;x^{2np+j}R(x)(1+x^{2p}+\dots+x^{2(n-1)p})+\frac{R(x)x^{2np+j}}{x^{p}-1}\nonumber\\
        &+x^{p}R^*(x)(1+x^{2p}+\dots+x^{2(n-1)p})-\frac{R^*(x)x^{2np}}{x^{p}-1}.
    \end{align}

    Upon substitution of $h(x)$ back into formula \ref{R(x)Neg1.3} for $R_{2np+j}(x)$ we find;
 
    \begin{align}
        R_{2np+j}(x) =&\; x^{2np+j+p}R(x)(1+x^{2p}+\dots+x^{2(n-1)p})+R^*(x)(1+x^{2p}+\dots+x^{2(n-1)p})\nonumber\\
        & -x^{2np+j}R(x)(1+x^{2p}+\dots+x^{2(n-1)p})-x^{p}R^*(x)(1+x^{2p}+\dots+x^{2(n-1)p})\nonumber\\
        & +x^{2np+j}R(x)(1+x^{2p}+\dots+x^{2(n-1)p})+ x^{p}R^*(x)(1+x^{2p}+\dots+x^{2(n-1)p})\nonumber\\
        & + \frac{R(x)x^{2np+j}-R^*(x)x^{2np}}{x^{p}-1}\\
        =&\; x^{2np+j+p}R(x)(1+x^{2p}+\dots+x^{2(n-1)p})+R^*(x)(1+x^{2p}+\dots+x^{2(n-1)p})\nonumber\\
        & + \frac{R(x)x^{2np+j}-R^*(x)x^{2np}}{x^{p}-1}.
    \end{align}

    We know that this is an integer polynomial as $\frac{R(x)x^{j}-R^*(x)x}{x^{p}-1}$ is also an integer polynomial by assumption. Hence, $R_{2np+j}(x)$ must be the companion polynomial for a representation of $1$ under the base $\alpha_{2np+j}^-$. This is exactly the companion polynomial for the string $\mathbf{a_{(n,j)}}$, it remains to show that it is the greedy expansion. 
    
    It suffices to apply Theorem \ref{greedy} to complete the proof. The result follows in the same fashion as Theorem \ref{MAIN}.
\end{proof}

\begin{theorem}
    Let $q$ be a Pisot number with minimal polynomial $M(x)$. Suppose that $1$ has an infinite reversibly greedy $\beta$-expansion in base $q$, say $a_1a_2\dots a_k(a_{k+1}\dots a_{k+\ell})^{\omega}$. Let $\kappa=a_2\dots a_{k}$ and $\gamma=a_{k+1}\dots a_{k+\ell}$
    
    Furthermore, assume that the companion polynomial of $q$ has co-factor polynomial $Q(x)$, which is reciprocal. Let $\alpha_{n\ell+j}^-$ be the Salem number that satisfies the polynomial $M(x)x^{n\ell+j}-M^*(x)$ for some $j\geq 1$.
    
    Suppose that the $\beta$-expansion for $\alpha_{\ell+j}^-$ is given by
    $$\mathbf{a_{(1,j)}}=a_1(a_2\dots a_k a_{k+1}\dots a_{k+\ell}\lambda_j a_{k}\dots a_{2}00 )^{\omega}=1(\kappa\gamma\lambda_j\kappa^*00)^{\omega},$$
    for some string $\lambda_j$ of length $j-k-1$,

    Then for $n\geq 1$ the $\beta$-expansion for $\alpha_{n\ell+j}^-$ is given by
    $$\mathbf{a_{(n,j)}}=a_1(a_2\dots a_k (a_{k+1}\dots a_{k+\ell})^{n}\lambda_j a_{k}\dots a_{2}00 )^{\omega}=1(\kappa\gamma^{n}\lambda_j\kappa^*00)^{\omega}.$$
    \label{MAIN NEG INF2}
\end{theorem}

\begin{proof}
    Let $R_{n\ell+j}(x)=Q(x)(M(x)x^{n\ell+j}-M^*(x))/(x^{\ell}-1)$ we show that if $R_{\ell+j}(x)$ is the companion polynomial to the $\beta$-expansion of $\alpha_{\ell+j}$ then $R_{n\ell+j}(x)$ must be the companion polynomial for the $\beta$-expansion of $\alpha_{n\ell+j}$.

    \begin{align}
        R_{n\ell+j}(x) &= (R(x)x^{n\ell+j}-R^*(x))/(x^{\ell}-1)\label{R(x)Neg2.1}\\
        &=\frac{R(x)x^j(x^{n\ell}-1)}{x^{\ell}-1}+\frac{R(x)x^j-R^*(x)}{x^{\ell}-1}\\
        &= R(x)x^j(1+x^{\ell}+\cdots+x^{(n-1)\ell})+A(x).
    \end{align}

    We look now specifically at the polynomial $A(x)$ which is not dependent on $n$, this is the polynomial whose coefficients will determine the string $\lambda_j\kappa^*00$ in the representation.
    \begin{align}
        A(x) &=\frac{R(x)x^j-R^*(x)}{x^{\ell}-1}\\
             &=\frac{x^jR(x)-x^{\ell}R^*(x)}{x^{\ell}-1}+\frac{x^{\ell}R^*(x)-R^*(x)}{x^{\ell}-1}\\
             &= \frac{x^jR(x)-x^{\ell}R^*(x)}{x^{\ell}-1} + R^*(x)\\
             &= \frac{x^jR(x)-x^{\ell}R^*(x)}{x^{\ell}-1} -x^{k+1}S^*(x)-x^{\ell}P_k^*(x)+ P_k^*(x).
    \end{align}

    We see that $A(x)$ is a polynomial by assumption and that $R_{n\ell+j}(x)$ has the desired form, where the coefficients of $A(x)$ up to $x^j$ exactly determine $\lambda_j\kappa^*00$.

    We finish by applying Theorem \ref{greedy} and noting that $a_1a_2\dots a_k(a_{k+1}\dots a_{k+\ell})^{\omega}$ is reversibly greedy to ensure that the require property holds.
\end{proof}

We see that Theorems \ref{MAIN NEG INF} and \ref{MAIN NEG INF2} prove that both patterns we saw in Example 3.1 are not coincidences and can be found, in general.

\begin{rmk}
    If we try to use a period of non-minimal length with an expansion in the form of Theorem \ref{MAIN NEG INF2} then we require an extra factor of $(x^{n\ell+2k+|\lambda_j|}+1)$ in our companion polynomial. This is because $x^{n\ell+2k+|\lambda_j|}+1$ is the length of the period. 
    
    Returning to Theorem \ref{MAIN NEG INF}, we see an extra factor of $(x^{n\ell+\ell/2}+1)$ in Equation \ref{R(x)Neg1.1}. Then we can see that in the case where $2k+|\lambda_j|\neq \ell/2$ these factors are not equal, and hence we cannot apply Theorem \ref{MAIN NEG INF}.
\end{rmk}

\subsection{Finite Case}
We consider the case where the greedy $\beta$-expansion under the Pisot base is finite. Let $q$ be the Pisot root of $M(x)$, and let the greedy $q$-expansion for $1$ be given by $a_1\dots a_k$.

\begin{example}
    Let $M(x)=x^3-2x^2+x-1=\Phi_2(x)$. We calculate the $\beta$-expansions for the Salem numbers defined by $M(x)x^m-M^*(x)$ for certain $m$ in Table \ref{neg_exp}. Note that the $\beta$-expansion for the Pisot number is given by the string $1101(0)^{\omega}$ and we have $k=4$. We let $\gamma=1100$ and $\kappa=100$. Notice that $(1100)^{\omega}=1\kappa\gamma^{\omega}=\gamma^{\omega}$ is exactly the quasi-greedy expansion. We will explore why this is the case in Remark \ref{finite quasi}. Table \ref{neg_exp} gives the greedy expansions for 1 under the Salem root of $M(x)x^m-M^*(x)$ for certain values of $m$. Note that the expansions for $m=4n+2$ do not appear to be periodic when calculating up periods of length 100.

    \begin{table}[h]
        \centering
        \begin{tabular}{l|l|l}
            $m$ & Expansion & Expansion\\ \hline
            5 & $1(1000000100)^{\omega}$ & $1(\kappa00\kappa^*00)^{\omega}$\\
            9 & $1(100110000001100100)^{\omega}$ & $1(\kappa\gamma00\gamma^*\kappa^*00)^{\omega}$\\
            13 & $1(10011001100000011001100100)^{\omega}$ & $1(\kappa\gamma\gamma00\gamma^*\gamma^*\kappa^*00)^{\omega}$\\ \hline
            7 & $1(100)^{\omega}=1(100100)^{\omega}$ & $1(\kappa)^{\omega}=1(\kappa100)^{\omega}$\\
            11 & $1(1001100100)^{\omega}$ & $1(\kappa\gamma100)^{\omega}$\\
            15 & $1(10011001100100)^{\omega}$ & $1(\kappa\gamma\gamma100)^{\omega}$\\ \hline
            8 & $1(1001010100100)^{\omega}$ & $1(\kappa10101\kappa^*00)^{\omega}$\\
            12 & $1(100110010101001100100)^{\omega}$ & $1(\kappa\gamma10101\gamma^*\kappa^*00)^{\omega}$\\
            16 & $1(10011001100101010011001100100)^{\omega}$ & $1(\kappa\gamma\gamma10101\gamma^*\gamma^*\kappa^*00)^{\omega}$\\
        \end{tabular}
        \caption{$\beta$-expansions for Salem numbers related to $\Phi_2$}
        \label{neg_exp}
    \end{table}

    Notice that as in the positive periodic case we have what looks to be a repeating cycle of fixed period, in this case the period is $4$ which happens to be the length of the $\beta$-expansion for the Pisot number. Note also that this is the length of the periodic part of the quasi-greedy expansion. Notice that by considering these expansions with non-minimal length we have that $n=7, 11, 15$ appear to have the same form as the other two sets as shown in Table \ref{n=7}. We will later show that this is not always the case.

    \begin{table}[h]
        \centering
        \begin{tabular}{l|l|l}
            $m$ & Expansion & Expansion\\ \hline
            7 & $1(100100100100)^{\omega}$ & $1(\kappa1001\kappa^*00)^{\omega}$\\
            11 & $1(10011001001001100100)^{\omega}$ & $1(\kappa\gamma1001\gamma^*\kappa^*00)^{\omega}$\\
            15 & $1(1001100110010010011001100100)^{\omega}$ & $1(\kappa\gamma\gamma1001\gamma^*\gamma^*\kappa^*00)^{\omega}$\\
        \end{tabular}
        \caption{$\beta$-expansions for Salem numbers related to $\Phi_2$}
        \label{n=7}
    \end{table}
    \label{negative exmp}
\end{example}

\begin{rmk}
    Before proceeding we note that in the above example the quasi-greedy expansion for $1$ under the Pisot base is given by $1100(1100)^{\omega}$ or $(1100)^{\omega}$. Applying Theorem \ref{MAIN NEG INF} and Theorem \ref{MAIN NEG INF2} to $1100(1100)^{\omega}$ yields the expansions in the above tables. We will later remark that this is not a coincidence. It follows from the fact that the quasi-greedy expansion which is constructed from a reversibly greedy expansion continues to satisfy Condition \ref{revgre}.
    \label{finite quasi}
\end{rmk}

We begin by considering the case where the minimal greedy expansion has a symmetric periodic part when we ignore the trailing zeros. In the above example, these are the values $m=4n+1$ and $m=4n+4$.

\begin{theorem}
    Let $q$ be a Pisot number with minimal polynomial $M(x)$. Let $q$ have a finite reversibly greedy $\beta$-expansion, say $a_1a_2\dots a_k$ with $k$ even. Let $k=2p$, $\kappa=a_2\dots a_{2p-1}(a_{2p}-1)$ and $\gamma=a_1\dots a_{2p-1}(a_{2p}-1)$. 
    
    Furthermore, assume that the companion polynomial of $q$ has co-factor polynomial $Q(x)$, which is reciprocal. Let $\alpha_{2np+j}^-$ be the Salem number that satisfies the polynomial $M(x)x^{2np+j}-M^*(x)$ for some $j\geq 1$. 
    
    Suppose that the $\beta$-expansion for 1 in base $\alpha_{2p+j}^-$ is given by
    \begin{align}
        \mathbf{a_{(1,j)}} &=a_1(a_2\dots a_{2p-1}(a_{2p}-1)\tau_j(a_{2p}-1)a_{2p-1}\dots a_200)^{\omega},\\
        &=1(\kappa\tau_j\kappa^*00)^{\omega}
    \end{align}

    For some string $\tau_j$ of length $j-1+p$.

    Then for $n\geq 1$ the $\beta$-expansion for 1 in base $\alpha_{2np+j}^-$ is given by
    \begin{align}
        \mathbf{a_{(n,j)}}=&\;a_1(a_2\dots a_{2p-1}(a_{2p}-1)(a_1\dots a_{2p-1}(a_{2p}-1))^{n-1}\tau_j\nonumber\\
        &\;((a_{2p}-1)a_{2p-1}\dots a_1)^{n-1}(a_{2p}-1)a_{2p-1}\dots a_200)^{\omega},\\
        =&\;\mathbf{a_{(n,j)}}=1(\kappa\underbrace{\gamma\gamma\dots\gamma}_{n-1}\tau_j\underbrace{\gamma^*\gamma^*\dots\gamma^*}_{n-1}\kappa^*00)^{\omega}.
    \end{align}

    \label{MAIN NEG}
\end{theorem}

\begin{proof}
    Let $R_{2np+j}(x)=Q(x)(M(x)x^{2np+j}-M^*(x))\left(\frac{x^{2np+p}+1}{x^{2p}-1}\right)$ we will show that if $R_{2p+j}(x)$ is the companion polynomial to the $\beta$-expansion of $\alpha_{2p+j}^-$ then $R_{2np+j}(x)$ must be the companion polynomial for the $\beta$-expansion of $\alpha_{2np+j}^-$.

    We need to show that the minimal polynomial of $\alpha_{2np+j}^-$, which divides $M(x)x^{2np+j}-M^*(x)$, also divides $R_{2np+j}(x)$. We also need to show it is the companion polynomial of the representation in the statement of the theorem. Towards this conclusion, let $R(x)=Q(x)M(x)$ be the companion polynomial of the greedy $\beta$-expansion for $1$ under base $q$, the Pisot number in question.

    First note that 
    \begin{equation}
        \frac{x^{2np+p}+1}{x^{2p}-1}=\frac{1-x^{p}+x^{2p}-x^{3p}+x^{4p}-\dots-x^{2np-p}+x^{2np}}{x^{p}-1}.
    \end{equation}
    This rational function has a numerator which can be written as,
    \begin{equation}
      f(x)=(x^{p}-1)(1+x^{2p}+\dots+x^{2(n-1)p})(x^{p}-1)+g(x)  
    \end{equation}
    where,
    \begin{equation}
        g(x)=x^{p}-x^{2p}+x^{3p}-x^{4p}+\dots+x^{2np-p}.
    \end{equation}

    Then using this representation we have,
    \begin{align}
        R_{2np+j}(x)=&\; (R(x)x^{2np+j}-R^*(x))\left(\frac{x^{2np+p}+1}{x^{2p}-1}\right)\\
        = &\; x^{2np+j+p}R(x)(1+x^{2p}+\dots+x^{2(n-1)p})+R^*(x)(1+x^{2p}+\dots+x^{2(n-1)p})\nonumber\\
        & -x^{2np+j}R(x)(1+x^{2p}+\dots+x^{2(n-1)p})-x^{p}R^*(x)(1+x^{2p}+\dots+x^{2(n-1)p})\nonumber\\
        & +(R(x)x^{2np+j}-R^*(x))\left(\frac{g(x)}{x^{p}-1}\right).\label{R(x)Neg2.2}
    \end{align}

    We will now focus on the final term of the polynomial in this form, that is $h(x)=(R(x)x^{2np+j}-R^*(x))\left(\frac{g(x)}{x^{p}-1}\right)$.
    We can write $g(x)$ in the following two forms,
    \begin{align}
        g(x) & = (x^{p}-1)(1+x^{2p}+\dots+x^{2(n-1)p})+\frac{1}{x^{p}-1}\\
        & = -x^{p}(x^{p}-1)(1+x^{2p}+\dots+x^{2(n-1)p})+\frac{x^{2np}}{x^{p}-1}.
    \end{align}

    Using the appropriate representation for $g(x)$ as we expand $h(x)$ we find,
    \begin{align}
        h(x) = &\; (R(x)x^{2np+j}-R^*(x))\left(\frac{g(x)}{x^{p}-1}\right)\\
        = &\; x^{2np+j}R(x)(1+x^{2p}+\dots+x^{2(n-1)p})+\frac{R(x)x^{2np+j}}{x^{p}-1}\nonumber\\
        & +x^{p}R^*(x)(1+x^{2p}+\dots+x^{2(n-1)p})-\frac{R^*(x)x^{2np}}{x^{p}-1}.
    \end{align}

    Upon substitution of $h(x)$ back into formula \ref{R(x)Neg2.2} for $R_{2np+j}(x)$ we find.

    \begin{align}
        R_{2np+j}(x) = &\; x^{2np+j+p}R(x)(1+x^{2p}+\dots+x^{2(n-1)p})+R^*(x)(1+x^{2p}+\dots+x^{2(n-1)p})\nonumber\\
        & -x^{2np+j}R(x)(1+x^{2p}+\dots+x^{2(n-1)p})-x^{p}R^*(x)(1+x^{2p}+\dots+x^{2(n-1)p})\nonumber\\
        & +x^{2np+j}R(x)(1+x^{2p}+\dots+x^{2(n-1)p})+ x^{p}R^*(x)(1+x^{2p}+\dots+x^{2(n-1)p})\nonumber\\
        & + \frac{R(x)x^{2np+j}-R^*(x)x^{2np}}{x^{p}-1}\\
        = &\; x^{2np+j+p}R(x)(1+x^{2p}+\dots+x^{2(n-1)p})+R^*(x)(1+x^{2p}+\dots+x^{2(n-1)p})\nonumber\\
        & + \frac{R(x)x^{2np+j}-R^*(x)x^{2np}}{x^{p}-1}.
    \end{align}

    We see, as $n$ increases, the repeated copies of $R(x)$ and $R^*(x)$ are the leading and trailing coefficients of our polynomial as expected. The middle terms of the expansion, dictated by $B(x)=(R(x)x^{j}-R^*(x))/(x^{p}-1)$, are independent of $n$. We notice in particular that $B(x)$ is symmetric with leading coefficient $1$. This is the reason we see $a_{2p}-1$ appearing in the expansion and not simply $a_{2p}$.

    We now prove $R_{2np+j}(x)\in \mathbb{Z}[x]$. It suffices to show $(R(x)x^{2np+j}-R^*(x)x^{2np})/(x^{p}-1)\in\mathbb{Z}[x]$. Under our assumption we have $(R(x)x^{2p+j}-R^*(x)x^{2p})/(x^{p}-1)\in\mathbb{Z}[x]$. Hence $(R(x)x^j-R^*(x))/(x^{p}-1)\in\mathbb{Z}[x]$, but $(R(x)x^{2np+j}-R^*(x)x^{2np})/(x^{p}-1)=x^{2np}(R(x)x^j-R^*(x))/(x^{p}-1)$ and so it must be an integer polynomial.

    Therefore $\mathbf{a_{(n,j)}}$ is a valid $\beta$-expansion for $1$ under base $\alpha_{2np+j}^-$, it remains to show it is the greedy expansion. This fact follows because the Pisot expansion, $a_1a_2\dots a_{2p}$, is reversibly greedy and we assume $\mathbf{a_{(1,j)}}$ is the greedy expansion for $\alpha_{2p+j}^-$ as in the proof of Theorem \ref{MAIN}. 
\end{proof}

\begin{rmk}
  Recall that if the greedy representation is given by $a_1\dots a_k$ then the quasi-greedy expansion is given by $a_1\dots a_{k-1}(a_k-1)(a_1\dots a_{k-1}(a_k-1))^{\omega}$. We see that by setting $\ell=k$, this expansion is exactly in the form required by Theorem \ref{MAIN NEG INF}. As such we see that the finite case is a consequence of the periodic case where we instead use the quasi-greedy expansion, this is proper as we know that a finite greedy expansion is reversibly greedy only if its quasi counterpart is as well, by Proposition \ref{quasi}.  
\end{rmk}

Now consider the other type of expansion seen in Example \ref{negative exmp}, the case $m=4n+3$ in that example.

\begin{theorem}
    Let $q$ be a Pisot number with minimal polynomial $M(x)$. Furthermore, let $q$ have a finite reversibly greedy $\beta$-expansion, say $a_1a_2\dots a_k$ with no restriction on $k$. Assume that the companion polynomial of $q$ has co-factor polynomial $Q(x)$, which is reciprocal. Let $\alpha_{nk+j}^-$ be the Salem number that satisfies the polynomial $M(x)x^{nk+j}-M^*(x)$ for some $j\geq 1$. Let $\kappa=a_2\dots a_{k-1}(a_k-1)$ and $\gamma=a_1\dots a_{k-1}(a_k-1)$.
    
    Suppose that the $\beta$-expansion for 1 in base $\alpha_{k+j}^-$ is given by
    \begin{equation}
        \mathbf{a_{(1,j)}}=a_1(a_2\dots a_{k-1}(a_k-1)\lambda_j00)^{\omega}=1(\kappa\lambda_j00)^{\omega},
    \end{equation}
    for some string $\lambda_j$ of length $j-2$.

    Then for $n\geq 1$ the $\beta$-expansion for 1 in base $\alpha_{nk+j}^-$ is given by
    \begin{align}
        \mathbf{a_{(n,j)}}&=a_1(a_2\dots a_{k-1}(a_k-1)(a_1\dots a_{k-1}(a_k-1))^{n-1}\lambda_j00)^{\omega}\\
        &=1(\kappa\gamma^{n-1}\lambda_j00)^{\omega}.
    \end{align}
    
    \label{MAIN NEG2}
\end{theorem}

\begin{proof}
    Consider the polynomial,
    \begin{equation}
        R_{nk+j}(x)=Q(x)(M(x)x^{nk+j}-M^*(x))/(x^k-1).
    \end{equation}
    We show that if $R_{k+j}(x)$ is the companion polynomial to the $\beta$-expansion of $\alpha_{k+j}$ then $R_{nk+j}(x)$ must be the companion polynomial for the $\beta$-expansion of $\alpha_{nk+j}$. Notice

    \begin{align}
        R_{nk+j}(x) &= (R(x)x^{nk+j}-R^*(x))/(x^k-1)\\
        &=\frac{R(x)x^j(x^{nk}-1)}{x^k-1}+\frac{R(x)x^j-R^*(x)}{x^k-1}\\
        &= R(x)x^j(1+x^k+\cdots+x^{(n-1)k})+R_{j}(x).
    \end{align}

    We see that under our assumption $R_{nk+j}(x)$ is indeed a polynomial as $R_{j}(x)=R_{k+j}(x)-R(x)x^j$ is a polynomial by assumption. This is indeed the companion polynomial of the desired expansion. This is because the successive powers of $R(x)$ in the companion polynomial correspond exactly to successive copies of $\gamma$ in the expansion it represents. The highest power representing $\kappa$ is an artifact of the periodic nature of the expansion.

    When applying Theorem \ref{greedy} we see that this is indeed a greedy expansion for 1 since $a_1a_2\dots a_k$ is greedy and we have $a_1a_2\dots a_k > \lambda_j$ by assumption.

\end{proof}

\begin{rmk}
    It appears this theorem  does not require the condition of being reversibly greedy. In the previous theorems the only part of the proof which required the Pisot expansion to be reversibly greedy is in using Theorem \ref{greedy} to prove that our expansion were themselves greedy. As the expansion does now include the inverted Pisot expansion it appears that revversibly greedy is not required. We are yet to find any examples that support this idea. One reason for this may be the requirement of having reciprocal co-factor.
\end{rmk}

\begin{rmk}
    As was the case in the previous theorem, we can view this finite case as a special case of the periodic case, Theorem \ref{MAIN NEG INF2}. We know that the quasi-greedy expansion is given by $a_1\dots a_{k-1}(a_k-1)(a_1\dots a_{k-1}(a_k-1))^{\omega}$. We can view this instead as the fully periodic expansion $(a_1\dots a_{k-1}(a_k-1))^{\omega}$. In this case, we have $\kappa$ being empty and $\gamma= a_1\dots a_{k-1}(a_k-1)$. Applying Theorem \ref{MAIN NEG INF2} our result follows directly. This is valid, as the quasi-greedy expansion is reversibly greedy whenever the finite greedy expansion is as well.
\end{rmk}

\begin{rmk}
    As in the periodic case, at first glance it appears that an expansion of the form seen in Theorem \ref{MAIN NEG2} can be viewed as in Theorem \ref{MAIN NEG} when we consider periods of longer length. We attempt this by appending a factor of $(x^{nk+|\lambda_j|+1}+1)$ to the companion polynomial. If we were able to apply Theorem \ref{MAIN NEG} then this factor would need to equal $(x^{nk+k/2}+1)$. Since $|\lambda_j|=j-2$ we would require $j=k/2+1$, as was the case in Example 3.5, which is not guaranteed.
\end{rmk}

\section{Examples and Motivation}

\subsection{Structure of Pisot Numbers}
Now that we have shown these results, we aim to show that the conditions of the theorems are relatively common occurrences. We will demonstrate this by finding a number of infinite families with these properties. We will first discuss the structure of Pisot numbers; we begin by considering the limit points of the Pisot numbers smaller than 2. The following result is due to Amara \cite{amara}.

\begin{theorem}
    The limit points of the Pisot numbers in $(1,2)$ are the following:
    \begin{equation}
        \varphi_1=\psi_1<\varphi_2<\psi_2<\varphi_3<\chi<\psi_3<\varphi_4<\cdots <\psi_r<\varphi_{r+1}<\cdots <2.
        \label{Limit Pisot}
    \end{equation}
    The minimal polynomial of $\varphi_r$ is $\Phi_r(x)=x^{r+1}-2x^r+x-1$, the minimal polynomial of $\psi_r$ is $\Psi_r(x)=x^{r+1}-x^r-\cdots -x-1$, and the minimal polynomial of $\chi$ is $\chi(x)=x^4-x^3-2x^2+1$.
\end{theorem}

The families of Pisot numbers that approach these limit points are known and are called regular Pisot numbers, and are given in Table \ref{Regular Pisot}. These polynomials are not generally irreducible; they do however admit a single root greater than 1, which is the regular Pisot number. The remaining factors of these polynomials are always cyclotomic.

Pisot numbers that are not regular are called irregular. It is known that for each Pisot number in Equation \ref{Limit Pisot} there exists $\varepsilon>0$ such that all Pisot numbers within $\varepsilon$ of this limit point are regular Pisot numbers. Further, it is known that only finitely many irregular Pisot numbers exist in the interval $[1,2-\epsilon]$ for any $\epsilon>0$, \cite{talmoudi}.

\begin{table}[]
    \centering
    \begin{tabular}{l|l}
        \hline
        Limit Points & Defining Polynomials \\ \hline
        $\varphi_r$ & $\Phi_{(r,q)}^{A\pm}(x)=\Phi_r(x)x^{q}\pm (x^r-x^{r-1}+1)$\\
        & $\Phi_{(r,q)}^{B\pm}(x)=\Phi_r(x)x^{q}\pm (x^r-x^+1)$\\
        & $\Phi_{(r,q)}^{C\pm}(x)=\Phi_r(x)x^{q}\pm (x^r+1)(x-1)$ \\
        $\psi_r$ & $\Psi_{(r,q)}^{A\pm}(x)=\Psi_r(x)x^{q}\pm (x^{r+1}-1)$ \\
        & $\Psi_{(r,q)}^{B\pm}(x)=\Psi_r(x)x^{q}\pm (x^r-1)/(x-1)$\\
        $\chi$ & $\chi_A^{\pm}(x)=\chi(x)x^{q}\pm (x^3+x^2-x-1))$ \\
        & $\chi_B^{\pm}(x)=\chi(x)x^{q}\pm (x^4-x^2+1))$
    \end{tabular}
    \caption{Regular Pisot Polynomials}
    \label{Regular Pisot}
\end{table}

Table \ref{PisotRep} contains the greedy $\beta$-expansions for most of the limit points and regular Pisot numbers. It also indicates whether they satisfy the reversibly greedy condition and gives the \textbf{pseudo-co-factor} associated with the given expansion. Recall that the co-factor of a representation is the polynomial $Q(x)$ such that $Q(x)M(x)=R(x)$, where $M(x)$ is the minimal polynomial for the base and $R(x)$ is the companion polynomial. The pseudo-co-factor is the polynomial $Q^*(x)$ such that the companion polynomial $R(x)$ can be written as $R(x)=M^*(x)Q^*(x)$, where $M^*(x)$ is the defining polynomials given from Table \ref{Regular Pisot} of the previous theorem, which may not be minimal. 

We see that most of these families have finite $\beta$-expansions. In fact they are finite reversibly greedy and hence covered by the previous result of Hare and Tweedle \cite{hareTweedle} in the positive case. The remaining six families have infinite beta-expansions and all have sub-families which are reversibly greedy. We will show a number of infinite sub-families that satisfy the conditions from our results.

\begin{sidewaystable}
    \fontsize{8pt}{8pt}
    \centering
    \begin{tabular}{l|l|l|l|l}
        Polynomial & Expansion & Restriction & Property & Pseudo Co-factor \\ \hline
        $\Phi_r(x)$ & $1^r0^{r-1}1$ & $r\geq 1$ &  Finite Reversibly Greedy & $(x^r-1)/(x-1)$\\ \hline
        $\Psi_r(x)$ & $1^{r+1}$ & $r\geq 1$ & Finite Reversibly Greedy & $1$ \\ \hline
        $\chi(x)$ & $11(10)^{\omega}$ & NA & Reversibly Greeedy & $1$ \\ \hline
        $\Phi_{(r,q)}^{A-}(x)$ & $1^r0^{r-1}10^{q-2r}10^r1^{r-1}$ & $q\geq2r$ & Finite Reversibly Greedy & $(x^r-1)/(x-1)$ \\
        $\Phi_{(r,q)}^{B-}(x)$ & $1^r0^{r-1}10^{q-2r}1^{r-1}0^r1$ & $q\geq2r$ & Finite Reversibly Greedy & $(x^r-1)/(x-1)$ \\
        $\Phi_{(r,q)}^{C-}(x)$ & $1^r0^{r-1}1(0^{q-2r}1^r0^r)^{\omega}$ & $q\geq2r+1$ & Reversibly Greedy & $(x^r-1)/(x-1)$ \\ \hline
        $\Phi_{(r,q)}^{A+}(x)$ & $(1^r0^r)^s1^{r-1}$ & $q=2rs+r-1$ & Finite Reversibly Greedy & $1/(x^r+1)(x-1)$ \\
        $\Phi_{(r,q)}^{B+}(x)$ & $(1^r0^r)^s1$ & $q=2rs+1$ & Finite Reversibly Greedy & $1/(x^r+1)(x-1)$ \\
        $\Phi_{(r,q)}^{C+}(x)$ & $(1^r0^r)^{s-1}0^{q-1}1$ & $q=2rs$ & Finite Reversibly Greedy & $(x^{q}-1)/(x^r+1)(x-1)$ \\ \hline
        $\Psi_{(r,q)}^{A-}(x)$ & $1^{r+1}(0^{q-r-1}1^r0)^{\omega}$ & $q\geq r$ & Reversibly Greedy & $1$\\
        $\Psi_{(r,q)}^{B-}(x)$ & $1^{r+1}0^{q-r}1^r$ & $q\geq r-1$ & Finite Reversibly Greedy & $1$ \\ \hline
        $\Psi_{(r,q)}^{A+}(x)$ & $(1^r0)^s0^{q-1}1$ & $q=(r+1)s$ & Finite Reversibly Greedy & $(1+x^{r+1}+\cdots+x^{(r+1)(s-1)})$ \\
        $\Psi_{(r,q)}^{B+}(x)$ & $(1^r0)^{s-1}1^r$ & $q (r+1)s+r$ & Finite Reversibly Greedy & $1/(x^{r+1}-1)$ \\ \hline
        $\chi_A^-(x)$ & $11(10)^{s-2}11011((10)^{s-2}0111(01)^{s-2}1000)^{\omega}$ & $q = 2s$ for $s\geq 2$ & Reversibly Greedy & $(x^{q}-1)(x^{q+1}+1)/(x^2-1)$ \\
        $\chi_A^-(x)$ & $11(10)^{s-2}11(00011(10)^{s-2}00)^{\omega}$ & $q = 2s+1$ for $s\geq 2$ & Reversibly Greedy & $(x^{q}-1)/(x^2-1)$ \\
        $\chi_A^+(x)$ & $11(10)^{s-2}011100(10)^{s-3}000010(00)^{s-2}11$ & $q = 2s$ for $s\geq 3$ & Finite Reversibly Greedy & $(x^{q-2}+1)(x^{q+2}-1)/(x^2-1)$ \\
        $\chi_A^+(x)$ & $11(10)^{s-1}01000(10)^{s-1}0(00)^s11$ & $q = 2s+1$ for $s\geq 1$ & Finite Reversibly Greedy & $(x^{q}+1)(x^{q+1}-1)/(x^2-1)$ \\ \hline
        $\chi_B^-(x)$ & $11(10)^{s-3}1100000(10)^{s-3}001$ & $q=2s$ for $s\geq 3$ & Finite Reversibly Greedy & $(x^{q-2}-1)/(x^2-1)$ \\
        $\chi_B^-(x)$ & $11(10)^{s-2}1101000(10)^{s-3}011(1(00)^{s-1}10)^{\omega}$ & $q=2s+1$ for $s\geq 3$ & Reversibly Greedy & $(x^{2q-2}-x^{q-1}-x^{q-3}+1)/(x^2-1)$ \\
        $\chi_B^+(x)$ & $11(10)^{s-2}0101(1(10)^{s-3}(011)^2(10)^{s-3}010^4100)^{\omega}$ & $q=2s$ for $s\geq 3$ & Reversibly Greedy & $(x^4(x^{q-1}+1)(x^2-x+1)+x^{2q+4}-1)/(x^2-1)$ \\
        $\chi_B^+(x)$ & $11(10)^{s-1}001$ & $q=2s+1$ for $s\geq 1$ & Finite Reversibly Greedy & $1/(x^2-1)$        
    \end{tabular}
    \caption{Expansions of Regular Pisot Numbers}
    Compiled in part from the work of Panju \cite{panju} and, Hare and Tweedle \cite{hareTweedle}.
    \label{PisotRep}
\end{sidewaystable}

\subsection{Positive Periodic Examples}
\begin{example}
    The beta-expansion for the Pisot polynomial $\Phi_{(r,\ell)}^{C-}(x)$ is given by $1^r0^{r-1}1(0^{\ell-2r}1^r0^r)^{\omega}$. We then calculate the constituent components of the companion polynomial and find the following;
    \begin{align}
        P_k(x)&=x^{2r}-x^{2r-1}-\cdots-x^r-1\\
        P^*_k(x)&=-x^{2r}-x^r-x^{r-1}-\cdots-x+1\\
        S(x)&=x^{2r-1}+x^{2r-2}+\cdots+x^r\\
        S^*(x)&=x^{\ell-1-r}+x^{\ell-1-(r+1)}+\cdots+x^{\ell-1-(2r-1)}
    \end{align}

    We then calculate the polynomial, $$Z_j(x)=(P_k(x)+S(x))x^{j+\ell n}+S(x)x^{j+\ell(n+1)}+P_k^*(x)x^{\ell (n+1)}+S^*(x)x^{\ell(n-1)+k+1}+S^*(x)x^{\ell n+k+1}.$$ We take $n=1$ which allows the coefficients of this polynomial to exactly determine the string $\tau_j$ in Theorem \ref{MAIN}.

    For example fix $r=2$ and $\ell=5$. Note that $\ell$ is the length of the periodic part of the Pisot expansion. Let $M(x)=\Phi_{(r,\ell)}^{C-}(x)$. Then we can calculate the beta-expansions for the Salem numbers with defining polynomials $M(x)x^{2\ell+j}+M^*(x)$ for each $1\leq j\leq 5$. We have
    \begin{align}
        P_k(x)&=x^4-x^3-x^2-1\\
        P^*_k(x)&=-x^4-x^2-x+1\\
        S(x)&=x^3+x^2\\
        S^*(x)&=x^2+x
    \end{align}

    We calculate the following polynomials,
    \begin{align}
        Z_1(x)&=x^{13}+2x^{10}+x^{7}\\
        Z_2(x)&=x^{15}+x^{11}+x^{10}+x^{6}\\
        Z_3(x)&=x^{16}+x^{15}-x^{14}+x^{12}+x^{10}-x^{8}+x^{7}+x^{6}\\
        Z_4(x)&=x^{17}+x^{16}-x^{14}+x^{13}+x^{10}-x^{9}+x^{7}+x^{6}\\
        Z_5(x)&=x^{18}+x^{17}+x^{7}+x^{6}
    \end{align}

    We can then use these polynomials to determine the strings $\tau_j$ for each $j$, we must be careful about leading and trailing zeroes due to the highest possible degree and the possible degree span of these polynomials.
    \begin{table}[!ht]
        \centering
        \begin{tabular}{l|l|l}
            j & $\tau_j$ & Notes\\ \hline
            1 & $00100200100$ & Invalid expansion\\
            2 & $010^3110^310$ & Valid expansion\\
            3 & $011(-1)01010(-1)110$ & Invalid expansion\\
            4 & $0110(-1)1001(-1)0110$ & Invalid expansion\\
            5 & $0110^9110$ & Valid expansion
        \end{tabular}
        \label{tab:my_label}
    \end{table}
    Here the coefficients of this polynomial for $j=1$ give the digits of a beta-expansion including a $2$ which is not an allowable digit. For the cases $j=3$ and $j=4$ we find negative coefficients which would result in negative digits in the beta-expansion. This shows us that $R_{n\ell+j}(x)$ is not the companion polynomial for the beta-expansion of the Salem number induced by the polynomial $M(x)x^{n\ell+j}+M^*(x)$ for $j=1,3,4$.
\end{example}

\subsection{Negative Periodic Examples}

\begin{example}
Consider the family of Pisot polynomials $\Phi^{C-}_{(r,q)}(x)=x^q\Phi_r(x)-(x^r+1)(x-1)$, where $\Phi_r(x)=x^{r+1}-2x^r+x-1$. The greedy $\beta$-expansion for $1$ under the Pisot root of $\Phi^{C-}_{(r,q)}(x)$ is given by $1^r0^{r-1}0(0^{q-2r}1^r0^r)^{\omega}$ for $q\geq 2r$ which is reversibly greedy. We show that for a specific value of $j$ the conditions of Theorem \ref{MAIN NEG2} are satisfied and we give the corresponding string $\lambda_j$. 

We know that the pseudo-co-factor is given by $Q^*(x)=\frac{(x^r-1)}{(x-1)}$ and we can write $R(x)=\Phi^{C-}_{(r,q)}(x)Q^*(x)$.
\begin{align}
    R(x) &=x^q\Phi^{C-}_{(r,q)}(x^r-1)/(x-1)-(x^r-1)(x^r+1)\\
         &=(x^{r+q+1}-2x^{r+q}+x^{q+1}-x^q)(x^{r-1}+\dots+x+1)-x^{2r}+1\\
         &=x^{2r+q}-x^{2r+q-1}-\dots-x^{r+1+1}-x^{r+1}-x^q-x^{2r}+1,
\end{align}
and
\begin{align}
    R^*(x) &= x^{2r+q}-x^q-x^{2r}-x^r-x^{r-1}-\dots-x+1.
\end{align}

From Theorem \ref{MAIN NEG INF2} we need to consider the function given by $A(x)=\frac{R(x)x^j-R^*(x)}{x^{\ell}-1}$. We show by computation that it is an integer polynomial and are then able to read the digits of $\lambda_j$ from its coefficients. Note that in this case we have $q=\ell$ and $k=2r$. Consider the case where $\ell\geq 3r$ and $j=\ell-r+1$.

\begin{align}
    A(x) =&\; (x^{2\ell+r+1}-x^{2\ell+r}-x^{2\ell+r-1}-\dots-x^{2\ell+1}-x^{2\ell-r+1}-x^{\ell+r+1}+x^{\ell-r+1}\nonumber\\
    &-x^{\ell+2r}+x^{\ell}+x^r+x^{r-1}+\dots+x-1)/(x^{\ell}-1)\\
    =&\; x^{\ell+r+1}-x^{\ell-r+1}-x^{2r}+1-(x^{\ell}+1)(x^r-x^{r-1}-\dots-x)\\
    =&\; x^{\ell+r+1}-x^{\ell+r}-\dots-x^{\ell+r}-x^{\ell-r+1}-x^{2r}-x^r-x^{r-1}-\dots-x+1. \label{phiC_A(x)}
\end{align}

The polynomial $-x^{2r}-x^r-x^{r-1}-\dots-x+1$ is exactly the part of the companion polynomial that corresponds to $\kappa^*00$. Note that for $\ell\geq3r$ we have $\ell-r+1\geq2r+1$ and so $A(x)$ only has coefficients in the set $\{0,\pm1\}$.

Since $\lambda_j$ has length $j-k-1=(\ell-r+1)-2r-1=\ell-3r$ the coefficients for powers of $x^{2r+i}$ for $i=1,\dots,\ell-3r$ correspond to the digits of $\lambda_j$.

We see from Equation \ref{phiC_A(x)} that these powers of $x$ all have a coefficient of $0$. We conclude that $\lambda_j=0^{\ell-3r}$, where the string is empty when $\ell=3r$. 
\end{example}

\begin{example}
We now move our focus to another family of Pisot numbers. Consider the family of polynomials $\Psi_{(r,q)}^{A-}(x)=x^q\Psi_r(x)-(x^{r+1}-1)$ where $\Psi_r(x)=x^{r+1}-x^r-\dots-x-1$. The greedy $\beta$-expansion for $1$ under the Pisot root of $\Psi_{(r,q)}^{A-}(x)$ is given by $1^{r+1}(0^{q-r-1}1^r0)^{\omega}$ for $q\geq r+1$, this is reversibly greedy.

The pseudo-co-factor is given by $Q^*(x)=1$ and so $R(x)=\Psi_{(r,q)}^{A-}(x)$. We will consider two sub-families.

First, the family $\Psi_{(1,q)}^{A-}(x)$ with $q\geq \max\{3,r+1\}$. We have 
\begin{equation}
    R(x)=x^{q+2}-x^{q+1}-x^q-x^2+1,
\end{equation}
with $\ell=q$ and $k=2$. And
\begin{equation}
    R^*(x)=x^{q+2}-x^q-x^2-x+1.
\end{equation}

We once again consider the polynomial $A(x)$ from Theorem \ref{MAIN NEG INF2} with $j=q$.
\begin{equation}
    A(x)= x^{q+2}-x^{q+1}-x^q-x^2-x+1.
\end{equation}
The length of $\lambda_j$ is $j-k-1=q-3$ which is non-negative. We can read the digits of $\lambda_j$ from the coefficients of $A(x)$. We conclude $\lambda_j=0^{q-3}$.

Second, the family $\Psi_{(1,2p)}^{A-}(x)$ with $p\geq \max\{2,(r+1)/2\}$. We have
\begin{equation}
    R(x)=x^{2p+2}-x^{2p+1}-x^{2p}-x^2+1,
\end{equation}
with $\ell=2p$ and $k=2$. And
\begin{equation}
    R^*(x)=x^{2p+2}-x^{2p}-x^2-x+1.
\end{equation}

In this case, we consider the function $B(x)=(R(x)x^j-R^*(x))/(x^p-1)$ with $j=p$ from Theorem \ref{MAIN NEG INF}.
\begin{equation}
    B(x)=x^{2p+2}-x^{2p+1}-x^{2p}-x^{p+1}-x^2-x+1.
\end{equation} This corresponds exactly to $\tau_j=0^{p-2}10^{p-2}$ which is exactly of the required length $j-k+p-1=p-2+p-1=2p-3$.

Finally,  a similar calculation shows that for $r\geq 2$ we have that $\Psi_{(r,2r)}^{A-}(x)$ satisfies the conditions of Theorem \ref{MAIN NEG INF2}. We find that for $j=2r$ that $\lambda_j=0^{r-2}$.
\end{example}

\subsection{Data}

In this section, we take a look at a data set that includes all Pisot numbers with minimal polynomial of degree at most 15. 

There are 555 Pisot numbers with a finite greedy $\beta$-expansion for 1 in the data set. 518 of which have a reversibly greedy expansion. There are 292 Pisot numbers with an eventually periodic greedy $\beta$-expansion for 1 in the data set. 254 of which have a reversibly greedy expansion. The following table summarizes the proportion of reversibly greedy Pisot numbers for which the theorems can be applied.

\begin{table}[]
    \centering
    \begin{tabular}{c|c|c}
        Case & Total Number of Pisot Numbers & Number of which satisfy a Theorem \\ \hline
        Finite Positive & 518 & 518\\
        Finite Negative & 518 & 186\\
        Periodic Positive & 254 & 163\\
        Periodic Negative & 254 & 182\\
    \end{tabular}
    \caption{Pisot Numbers for which our results can be applied}
    \label{tab:my_label}
\end{table}

Upon examining our dataset it becomes clear that there are patterns, we see some families of regular Pisot numbers that will always satisfy the conditions and other families that only sometimes satisfy the conditions. Of particular note is the Finite Negative case which appears to have no families that always satisfy the conditions, apart from the limit Pisot numbers themselves. Some of these can be seen in the following table, where each family works for large enough $r$ and $q$. Table \ref{Families of Regular Pisot Numbers} gives certain families of Pisot numbers that always satisfy the conditions of the relevant theorems, it also gives a value for $j$ that works in those cases.

\begin{table}[]
    \centering
    \begin{tabular}{c|c|c|c|c}
        Case & Family & $j$ & Family & $j$ \\ \hline
        Finite Positive & All Families & All $j$ & &\\
        Finite Negative & $\Phi_r$ & $j=r$ & $\Psi_r$ & $j=r-1$\\
        Periodic Positive & $\Phi_{(r,q)}^{C-}$ & $j=q$ & $\Psi_{(1,q)}^{A-}$ & $j=q$\\
        Periodic Negative & $\Phi_{(r,q)}^{C-}$ & $j=q-r+1$ & $\Psi_{(1,q)}^{A-}$ & $j=q$\\
    \end{tabular}
    \caption{Families of Regular Pisot Numbers}
    \label{Families of Regular Pisot Numbers}
\end{table}

\section{Further Remarks}

In this paper, we explore greedy $\beta$-expansions for Salem numbers which are related to Pisot numbers. Further work could include determining a criteria for exactly when these theorems apply. We give data-driven evidence to suggest why we expect a large proportion of Pisot numbers to have these properties. A deeper exploration of individual families of regular Pisot numbers as to when the results apply to sub-families would be interesting.

Other potential expansions of these results include considering Pisot numbers $\beta>2$, on the alphabet $\{0,1\dots,\lfloor\beta\rfloor\}$. There are a few partial results due to Hichri \cite{hichri} in the finite positive case, they explore degree 2 and 3 Pisot numbers. They also go over a specific case for larger degree Pisot numbers.

\end{document}